\documentclass[letterpaper, 10 pt, conference]{ieeeconf}  

\IEEEoverridecommandlockouts                              

\usepackage{graphicx} 
\usepackage{amsmath} 
\usepackage{amssymb}  
\usepackage{algorithm}
\usepackage{algpseudocode}

\usepackage[noadjust]{cite}

\newtheorem{lemma}{Lemma}
\newtheorem{theorem}{Theorem}

\newtheorem{condition}{Condition}
\newtheorem{proposition}{Proposition}
\newtheorem{remark}{Remark}
\newtheorem{corollary}{Corollary}

\DeclareMathOperator{\alim}{as-lim}
\DeclareMathOperator{\rank}{rank}
\DeclareMathOperator{\diag}{diag}
\DeclareMathOperator{\blkdiag}{blkdiag}
\DeclareMathOperator{\trace}{trace}

\DeclareMathOperator{\range}{range}

\title{\LARGE \bf
Statistical Watermarking for Networked Control Systems
}

\author{Pedro Hespanhol, Matthew Porter, Ram Vasudevan, and Anil Aswani
\thanks{This work was supported by the UC Berkeley Center for Long-Term Cybersecurity, and by Ford Motor Company.}
\thanks{Pedro Hespanhol and Anil Aswani are with the Department of Industrial Engineering and Operations Research, University of California, Berkeley, CA 94720, USA 
        {\tt\small pedrohespanhol@berkeley.edu, aaswani@berkeley.edu}}%
				\thanks{Matthew Porter and Ram Vasudevan are with the Department of Mechanical Engineering, University of Michigan, Ann Arbor, MI 48109, USA 
        {\tt\small matthepo@umich.edu, ramv@umich.edu}}%
}

\begin{document}

\maketitle
\thispagestyle{empty}
\pagestyle{empty}

\begin{abstract}
Watermarking can detect sensor attacks in control systems by injecting a private signal into the control, whereby attacks are identified by checking the statistics of the sensor measurements and private signal.  However, past approaches assume full state measurements or a centralized controller, which is not found in networked LTI systems with subcontrollers.  Since generally the entire system is neither controllable nor observable by a single subcontroller, communication of sensor measurements is required to ensure closed-loop stability.  The possibility of attacking the communication channel has not been explicitly considered by previous watermarking schemes, and requires a new design.  In this paper, we derive a statistical watermarking test that can detect both sensor and communication attacks.  A unique (compared to the non-networked case) aspect of the implementing this test is the state-feedback controller must be designed so that the closed-loop system is controllable by each sub-controller, and we provide two approaches to design such a controller using Heymann's lemma and a multi-input generalization of Heymann's lemma.  The usefulness of our approach is demonstrated with a simulation of detecting attacks in a platoon of autonomous vehicles.  Our test allows each vehicle to independently detect attacks on both the communication channel between vehicles and on the sensor measurements.
\end{abstract}


\section{Introduction}

A major challenge in designing networked control systems is ensuring resilience of subcontrollers to malicious attacks \cite{cardenas2008research,abrams2008malicious,langner2011stuxnet,cardenas2008research}.  Unlike systems with a centralized controller where sensor attacks occur only on the output measurements, networked control systems are also susceptible to attacks on the communication channel used to transfer measurements between subcontrollers.  Though cryptography and cybersecurity \cite{parno2006secure,kumar2006managing,wang2013cyber,kim2012cyber} can secure communication channels, many networked control system applications feature rapid reconfiguration of the network and cannot operate in real-time with the overhead required to establish secure communication channels.  For instance, consider a platoon of autonomous vehicles where vehicles rapidly enter or exit the platoon.

This paper develops a statistical watermarking approach for detecting malicious sensor and communication attacks on networked LTI systems. Our first contribution is to design a watermarking test using null hypothesis testing \cite{mo2009secure,mo2010false,mo2014detecting,weerakkody2014detecting,mo2015physical}, and this requires characterizing the statistics of states and private watermarking signals under the dynamics of multiple subcontrollers within the networked system.  A unique feature (as compared to the non-networked setting) of the watermarking scheme is it requires the state-feedback control to be such that the closed-loop system is controllable by each subcontroller, because otherwise a subcontroller could not independently verify the lack of an attack.  Our second contribution is to provide two approaches to constructing such a state-feedback control, and this partly involves deriving a multi-input generalization of Heymann's lemma \cite{heymann1968comments,hautus1977simple}.


\subsection{Watermarking for LTI Systems}

Watermarking been proposed for detecting sensor attacks in LTI systems \cite{satchidanandan2016dynamic,ko2016theory,mo2009secure,mo2010false,mo2014detecting,weerakkody2014detecting,mo2015physical,hespanhol2017dynamic}.  The idea is that a random signal is added to the control, and the statistics of the sensor measurements and random signal are compared to check for sensor attacks.  Some schemes use null hypothesis testing \cite{mo2009secure,mo2010false,mo2014detecting,weerakkody2014detecting,mo2015physical}, while dynamic watermarking \cite{satchidanandan2016dynamic,ko2016theory,hespanhol2017dynamic} ensures only attacks adding a zero-average-power signal to the sensor measurements can remain undetected.  Most watermarking \cite{mo2009secure,mo2010false,mo2014detecting,weerakkody2014detecting,mo2015physical,hespanhol2017dynamic} is for LTI systems with a centralized controller, and only the approaches of \cite{satchidanandan2016dynamic,ko2016theory} can be used with networked LTI systems; however, the approaches \cite{satchidanandan2016dynamic,ko2016theory} require full state observation, which is not the case for many systems.  Our first contribution is to develop watermarking for networked LTI systems with partial state observation.

\subsection{Security for Intelligent Transportation Systems}

Intelligent transportation systems (ITS) \cite{ko2016theory,gonzalez2010perpetual,aswani2011,zhang2012hierarchical,vasudevan2012safe,mohan2016convex,como2016convexity} may benefit from watermarking.  For instance, \cite{ko2016theory} considered the use of dynamic watermarking to detect sensor attacks in a network of autonomous vehicles coordinated by a supervisory controller; the watermarking approach was successfully able to detect attacks.  However, large-scale deployments of ITS must be resilient in the face of partial state observations and partially distributed control structures.  For example, vehicle platoons are susceptible to malicious interference of GPS and the communication channel between vehicles \cite{SHEPARD2012146,Nighswander:2012:GSA:2382196.2382245,6837385}.  Our third contribution is to conduct a simulation that shows the efficacy of our watermarking scheme in detecting attacks on a vehicle platoon.

\subsection{Outline}
Sect. \ref{sec:ltisam} provides a model of the networked LTI system we consider, and specifies a model for communication and sensor attacks.  Next, Sect. \ref{sec:iwd} presents an example to give intuition about the new challenges with designing watermarking for networked systems.  We construct a statistical watermarking test in Sect. \ref{sec:dct}, which allows each subcontroller to independently check for the presence of communication or sensor attacks.  Our tests require a state-feedback controller such that the closed-loop system is controllable by each subcontroller, and Sect. \ref{sec:sfbk} provides two methods for constructing such a controller.  Last, Sect. \ref{sec:sav} conducts simulations of an autonomous vehicle platoon.  We show that our approach is able to detect the presence or absence of sensor attacks and attacks on the communication channel between vehicles.

\section{Networked LTI System and Attack Modalities}

\label{sec:ltisam}

\subsection{Dynamics and Measurement Model}

We study a setting with $\kappa$ subcontrollers.  The subscripts $i$ or j denote the $i$-th or $j$-th subcontroller, and the subscript $n$ indicates time.  Consider the LTI system with dynamics
\begin{equation}
x_{n+1} = Ax_n + \textstyle\sum_{i=1}^\kappa B_iu_{i,n} + w_n,
\end{equation}
where $x \in\mathbb{R}^p$ is the state, $u_i \in \mathbb{R}^{q_i}$ is the input of the $i$-th subcontroller, and $w\in\mathbb{R}^p$ is a zero mean i.i.d. process noise with a jointly Gaussian distribution and covariance $\Sigma_W$.  Each subcontroller steers a subset of the actuators, and each subcontroller makes the partial state observations
\begin{equation}
y_{i,n} = C_ix_n + z_{i,n} + v_{i,n},
\end{equation}
where $y_{i} \in \mathbb{R}^{m_i}$ is the observation of the $i$-th subcontroller, $z_{i}\in\mathbb{R}^{m_i}$ is zero mean i.i.d. measurement noise with a jointly Gaussian distribution and covariance $\Sigma_{Z,I}$, and $v_{i}\in\mathbb{R}^{m_i}$ should be interpreted as an additive measurement disturbance that is added by an attacker.  

\subsection{Network Communication Model}

The LTI system here is networked in the following sense:  The dynamics and partial observations are such that for 
\begin{align}
&B = \begin{bmatrix} B_1 & \cdots & B_\kappa\end{bmatrix}\\
&C^\textsf{T} = \begin{bmatrix} C_1^\textsf{T} & \cdots & C_\kappa^\textsf{T}\end{bmatrix}
\end{align}
we have that $(A,B)$ is stabilizable and $(A,C)$ is detectable.  In general, $(A,B_i)$ is not stabilizable for some (or all) $i$, and similarly $(A, C_i)$ is not detectable for some (or all) $i$.  Thus coordination is required between subcontrollers to ensure closed-loop stability, and networking arises because we assume each subcontroller communicates its own partial state observations to all other subcontrollers.  (Our setting assumes communication has zero cost.)  Consider the values
\begin{equation}
s_{i,j,n} = y_{j,n} + \nu_{i,j,n},
\end{equation}
where $s_{i,j}\in\mathbb{R}^{m_j}$ is the value communicated to subcontroller $i$ of the measurement made by subcontroller $j$, and $\nu_{i,j,n}\in\mathbb{R}^{m_j}$ should be interpreted as an additive communication disturbance added by an attacker.  Clearly $\nu_{i,i,n} \equiv 0$ for all $i$, since the $i$-th subcontroller already has its own measurement.

\subsection{Controller and Observer Structure}

The idea of statistical watermarking in this context will be to superimpose a private (and random) excitation signal $e_{i,n}$ known in value to only the $i$-th subcontroller but unknown in value to the attacker or to the other subcontrollers.  We will apply the control input $u_{i,n} = K_i\hat{x}_{i,n} + e_{i,n}$, where $\hat{x}_{i,n}$ is the observer-estimated state (the subscript $i$ here indicates that each subcontroller operates its own observer, and that $\hat{x}_{i,n}$ is the state estimated by the observer of the $i$-th subcontroller) and $e_{i,n}$ are i.i.d. Gaussian with zero mean and constant variance $\Sigma_{E,I}$ fixed by the subcontrollers.

Now let $K_i$ be constant state-feedback gain matrices such that $A+\sum_{i=1}^\kappa B_iK_i$ is Schur stable, and let $L_i$ be constant observer gain matrices.  It will be useful to define
\begin{align}
&K^\textsf{T} = \begin{bmatrix} K_1^\textsf{T} & \cdots & K_\kappa^\textsf{T}\end{bmatrix}\\
&L = \begin{bmatrix} L_1 & \cdots & L_\kappa\end{bmatrix}
\end{align}
Then the closed-loop system with private excitation is
\begin{equation}
\label{eqn:clpe}
\begin{aligned}
&x_{n+1} = Ax_n + \textstyle\sum_{j=1}^\kappa B_j(K_j\hat{x}_{j,n} + e_{j,n}) + w_n\\
&\hat{x}_{i,n+1} = \textstyle(A+\sum_{j=1}^\kappa B_jK_j + \sum_{j=1}^\kappa L_jC_j)\hat{x}_{i,n} +\\
&\hspace{2cm}\textstyle-\sum_{j=1}^\kappa L_jC_jx_n + B_ie_{i,n} +\\
&\hspace{3cm}\textstyle-\sum_{j=1}^\kappa L_j(z_{j,n}+v_{j,n}+\nu_{i,j,n}) 
\end{aligned}
\end{equation}
These equations represent the fact that each subcontroller has its own observer using the measurements that it has received.  It is not clear \textit{a priori} that this closed-loop system is stable since each observer may start at a different initial condition.  This concern is resolved by the following result:

\begin{proposition}
\label{prop:cllp}
Let $K_i$ and $L_i$ be constant state-feedback and observer gains such that $A+BK$, $A+LC$, and $A+BK+LC$ are Schur stable.  The closed-loop system (\ref{eqn:clpe}) is Schur stable with no private excitation $e_{j,n} \equiv 0$, process noise $w_n\equiv 0$, measurement noise $z_{j,n} \equiv 0$, measurement attack $v_{j,n}\equiv 0$, and communication attack $\nu_{i,j,n}\equiv 0$.
\end{proposition}

\begin{proof}
Consider the change of variables from the states $x,\hat{x}_i$ to the states $x,\delta_1,d_i$ where $\delta_1 = \hat{x}_1 - x$ and $d_i = \hat{x}_i - \hat{x}_1$ for $i = 2,\ldots,\kappa$.  Then inserting this change of variables into (\ref{eqn:clpe}) gives
\begin{equation}
\label{eqn:dynamicschange2}
\begin{aligned}
&x_{n+1} = (A+\textstyle\sum_{j=1}^\kappa B_jK_j)x_n + (\sum_{j=1}^\kappa B_jK_j)\delta_{1,n} + \\
&\hspace{2cm}\textstyle\sum_{j=2}^\kappa B_jK_jd_{j,n}\\
&\delta_{1,n+1} = \textstyle(A+\sum_{j=1}^\kappa L_jC_j)\delta_{1,n} - \sum_{j=2}^\kappa B_jK_jd_{j,n}\\
&d_{i,n+1} = \textstyle(A+\sum_{j=1}^\kappa B_jK_j + \sum_{j=1}^\kappa L_jC_j)d_{i,n}, \\
&\hspace{5cm}\text{ for } i=2,\ldots,\kappa\\
\end{aligned}
\end{equation}
If we put $x,\delta_i$ into a single vector $\check{x}$, then the dynamics $\check{x}_{n+1} =\check{A}\check{x}_n$ are such that $\check{A}$ is a block upper-triangular matrix with $A+BK$, $A+LC$, and $A+BK+LC$ on the diagonal.  This means $\check{A}$ is Schur stable since we assumed $A+BK$, $A+LC$, and $A+BK+LC$ are Schur stable.
\end{proof}

\begin{remark}
This result implies that the separation principle does not hold.  Fortunately, this is not a substantial impediment from the standpoint of design.  Given a $K$ such that $A+BK$ is stable, we can solve an LMI formulation \cite{boyd1994linear,de1999new} to choose (when feasible) a $C$ such that both $A+LC$ and $A+BK+LC$ are Schur stable.  In particular, suppose there exists a positive definite matrix $Q\succ 0$ and general matrix $R$ such that the following two LMI's
\begin{equation}
\begin{aligned}
&\begin{bmatrix} Q & A^\textsf{T}Q+C^\textsf{T}R \\ Q^\textsf{T}A+R^\textsf{T}C & Q\end{bmatrix} \succ 0\\
&\begin{bmatrix} Q & (A+BK)^\textsf{T}Q+C^\textsf{T}R \\ Q^\textsf{T}(A+BK)+R^\textsf{T}C & Q\end{bmatrix} \succ 0
\end{aligned}
\end{equation}
are satisfied.  Then choosing $L = Q^{-1}R^\textsf{T}$ ensures that $A+LC$ and $A+BK+LC$ are Schur stable.  Convex optimization can be used to determine if these LMI's have a solution, and compute a solution if possible.
\end{remark}

For the purpose of designing our test, it will be useful to define another change of variables on the states.  Consider the change of variables from the states $x,\hat{x}_i$ to the states $x,\delta_i$ where $\delta_i = \hat{x}_i - x$ for $i = 2,\ldots,\kappa$.  If there is no measurement attack $v_{j,n}\equiv 0$ and no communication attack $\nu_{i,j,n}\equiv 0$, then a straightforward calculation gives
\begin{equation}
\label{eqn:dynamicschange}
\begin{aligned}
&x_{n+1} = (A+BK)x_n + \textstyle\sum_{j=1}^\kappa B_j(K_j\delta_{j,n} + e_{j,n}) + w_n\\
&\delta_{i,n+1} = \textstyle(A+BK+LC)\delta_{i,n} + B_ie_{i,n}-\sum_{j=1}^\kappa L_jz_{j,n}+\\
&\hspace{1cm}-\textstyle\sum_{j=1}^\kappa B_j(K_j\delta_{j,n} + e_{j,n}) -w_n 
\end{aligned}
\end{equation}
If we define $\Delta^\textsf{T} = \begin{bmatrix} \delta_1^\textsf{T} & \cdots & \delta_\kappa^\textsf{T}\end{bmatrix}$ and $E^\textsf{T} = \begin{bmatrix} e_1^\textsf{T} & \cdots & e_\kappa^\textsf{T}\end{bmatrix}$, then the above dynamics for the $\delta_i$ can be written as
\begin{multline}
\Delta_{n+1} = \underline{A}\Delta_n + \mathrm{blkdiag}(B_1,\ldots,B_\kappa)E_n+\\
\textstyle- \begin{bmatrix}1 & \cdots & 1\end{bmatrix}^\textsf{T}\otimes\left(-w_n + \sum_{j=1}^\kappa L_jz_{j,n} + B_je_{j,n}\right),
\end{multline}
where $\otimes$ is the Kronecker product, $\mathrm{blkdiag}(B_1,\ldots,B_\kappa)$ is the block diagonal matrix with $B_1,\ldots,B_\kappa$ on the diagonals, and $\underline{A}$ is the corresponding matrix defined to make the above equivalent to (\ref{eqn:dynamicschange}).  This will be used to define our test.

%
%
%
%
%
%

\section{Intuition for Watermarking Design}

\label{sec:iwd}
Watermarking for networked systems faces new challenges not encountered in the non-networked setting.  To illustrate the new difficulty, consider the networked LTI system with
\begin{equation}
A = \begin{bmatrix} 1 & 1 \\ 0 & 1\end{bmatrix}\  B_1 = \begin{bmatrix}1 \\ 0\end{bmatrix}\ B_2 = \begin{bmatrix} 0 \\ 1\end{bmatrix}\ 
\begin{aligned}
&C_1 = \begin{bmatrix} 1 & 0\end{bmatrix}\\
&C_2= \begin{bmatrix}0&1\end{bmatrix}
\end{aligned}
\end{equation}
In this example, $(A,B_1)$ is not stabilizable and $(A,C_2)$ is not detectable.  And so coordination is required between the subcontrollers to stabilize the system.  

For instance, the choice $K = -\frac{1}{2}\mathbb{I}$ makes $A+BK$ Schur stable, and implementing the corresponding output-feedback controller requires communication of partial observations between the two subcontrollers.  In this case, the design is such that the first subcontroller cannot inject any watermarking signal into the second state, while the second subcontroller cannot inject any watermarking signal into the first state.  This is problematic because this means each subcontroller cannot verify the accuracy of the communicated state information.

However, suppose we instead choose
\begin{equation}
K = -\frac{1}{2}\begin{bmatrix} 1 & 1 \\ 1 & 1\end{bmatrix}
\end{equation}
Then $A+BK$ is Schur stable.  More importantly, $(A+BK, B_1)$ and $(A+BK, B_2)$ are controllable with this $K$.  Thus each subcontroller can inject a private watermarking signal known only to the subcontroller, and such that this signal can be used to verify the accuracy of the communicated state information and of the partial observations.  

This example shows that designing watermarking differs in the networked and non-networked cases.  The networked case requires designing both the state-feedback controller and the corresponding tests to detect attacks; whereas watermarking in the non-networked case only requires designing the the corresponding tests to detect attacks \cite{mo2009secure,mo2010false,mo2014detecting,weerakkody2014detecting,mo2015physical,hespanhol2017dynamic}.

\section{Specification of Statistical Test}

\label{sec:dct}

Though watermarking for networked systems requires designing both the state-feedback controller and watermarking tests, we first focus on the latter.  We construct a statistical test using the framework of null hypothesis testing, after assuming the existence of a state-feedback controller satisfying: 
\begin{condition}
\label{cond:one}
Let $k_{i,j}' = \min\{k\geq 0\ |\ C_j(A+BK)^kB_i \neq 0\}$.  For each $i$ and $j$, there exists a $k_{i,j}' \leq p-1$.
\end{condition}

This condition is itself nontrivial because it may be that $C_j(A+BK)^kB_i \equiv 0$ for all $k \geq 0$.  Approaches to synthesize a state-feedback controller $K$ to ensure the above condition holds will be shown in the next section.  This property is important because it means the watermarking signal of the $i$-th subcontroller is seen in the $j$-th output when the system is controlled by perfect-information state-feedback.

\subsection{Variable Definitions}

Now before specifying the test, it is useful to define some variables.  Suppose we have $K$, $L$ such that $A+BK$, $A+LC$, and $A+BK+LC$ are Schur stable.  Let $\Sigma_{\Delta}$ be the positive semidefinite matrix that solves the Lyapunov equation
\begin{multline}
\Sigma_\Delta = \underline{A}\Sigma_\Delta\underline{A}^\textsf{T} + \\
\mathrm{blkdiag}(B_1,\ldots,B_\kappa)\Sigma_E\mathrm{blkdiag}(B_1,\ldots,B_\kappa)^\textsf{T} + \\
\textstyle- \begin{bmatrix}1 & \cdots & 1\\ \vdots & \vdots & \vdots \\ 1 & \cdots & 1\end{bmatrix}^\textsf{T}\otimes\Big(\Sigma_W + \sum_{j=1}^\kappa L_j\Sigma_{Z,J}(L_j)^\textsf{T} + \\
B_j\Sigma_{E,J}(B_j)^\textsf{T}\Big),
\end{multline}
where $\Sigma_E = \mathrm{blkdiag}(\Sigma_{E,1},\ldots,\Sigma_{E,\kappa})$. A solution exists because the above is a Lyapunov equation and since Proposition \ref{prop:cllp} ensured stability.  Note by construction $\Sigma_{\Delta} = \alim_N \frac{1}{N}\sum_{n=0}^{N-1} \Delta_{n}^{\vphantom{\textsf{T}}}\Delta_{n}^\textsf{T}$ when there is no attack (i.e., $v_{i,n}\equiv 0$ and $\nu_{i,j,n}\equiv 0$ for all $i,j,n$).  If we divide $\Sigma_\Delta\in\mathbb{R}^{\kappa p\times \kappa p}$ into sub-matrices with dimension $p\times p$, then define $D_I\in\mathbb{R}^{p\times p}$ to be the $i\times i$-th sub-matrix of $\Sigma_\Delta$.

Lastly, we consider the matrix dynamics
\begin{multline}
\mathbb{E}(\Delta_{n+1}e_{i,t}^\textsf{T}) = \underline{A}\mathbb{E}(\Delta_ne_{i,t}^\textsf{T})) + \\
f_i\otimes B_i\Sigma_{E,I}\cdot\mathbf{1}(t = n)+\\
\textstyle- \begin{bmatrix}1 & \cdots & 1\end{bmatrix}^\textsf{T}\otimes\left(B_i\Sigma_{E,I}\right)\cdot\mathbf{1}(t = n),
\end{multline}
where $\mathbf{1}(\cdot)$ is an indictor function, and the vector $f_i$ has a one in the $i$-th position and is zero otherwise.  This means \begin{multline}
\Sigma_{\Delta,I,k} := \mathbb{E}(\Delta_{n}e_{i,n-k-1}^\textsf{T}) = \underline{A}^kf_i\otimes B_i\Sigma_{E,I}+\\
\textstyle- \underline{A}^k\begin{bmatrix}1 & \cdots & 1\end{bmatrix}^\textsf{T}\otimes\left(B_i\Sigma_{E,I}\right).
\end{multline}
If we divide $\Sigma_{\Delta,I,k}\in\mathbb{R}^{\kappa p\times q_i}$ in sub-matrices of size $p\times q_i$, then let $Q_{I,J,k}\in\mathbb{R}^{p\times q_i}$ be the $j$-th sub-matrix of $\Sigma_{\Delta,I,k}$.

\subsection{Definition of Test}

\begin{algorithm}[t]
\caption{Compute State-Feedback $K$ for Proposition \ref{prop:mihl}}
\label{alg:mihl}
\begin{algorithmic}
\State $x_1 := \frac{b_1}{\|b_1\|}$
\ForAll{$k \in \{1,\ldots,p-1\}$}
\State $x_{k+1} := \lambda^k\frac{b_{k+1}}{\|b_{k+1}\|}$
\State $u_k := B^{-1}\big(x_{k+1}-Ax_k\big)$
\EndFor
\State $x_{p+1} := \lambda^p\frac{b_1}{\|b_1\|}$
\State $u_p := B^{-1}\big(x_{p+1}-Ax_p\big)$
\State $X := \begin{bmatrix} x_1 & \cdots & x_p\end{bmatrix}$
\State $U := \begin{bmatrix} u_1 & \ldots & u_p\end{bmatrix}$
\State $K := UX^{-1}$
\end{algorithmic}
\end{algorithm}

Our statistical watermarking test will involve the (second-order) statistical characterization of the vectors defined as
\begin{equation}
\psi_{n,i,j} = \begin{bmatrix} e_{i,n-k_{i,j}'-1} \\ C_j\hat{x}_{i,n} - s_{i,j,n}\end{bmatrix},
\end{equation}
and this characterization will be used to specify the distribution corresponding to the null hypothesis of no attack.  

\begin{theorem}
If we have that $v_{i,n}\equiv 0$ and $\nu_{i,j,n}\equiv 0$, then $\textstyle\alim_{N} \frac{1}{N}\sum_{n=0}^{N-1} \psi_n^{\vphantom{\textsf{T}}}\psi_n^\textsf{T} = R_{I,J}$, where
\begin{equation}
\label{eqn:acttestjoint}
R_{I,J} = \begin{bmatrix}\Sigma_{E,I} & Q_{I,J,k'_{i,j}}^\textsf{T}C_j^\textsf{T}\\ C_jQ_{I,J,k'_{i,j}} & C_jD_iC_j^\textsf{T}+\Sigma_{Z,J}\end{bmatrix}.
\end{equation}
Moreover, we have that $\alim_n \mathbb{E}(\psi_n) = 0$.
\end{theorem}

\begin{proof}
First note that we have $\alim_n \mathbb{E}(\delta_{i,n}) = 0$ by the stability from Proposition \ref{prop:cllp}.  But $C_j\hat{x}_{i,n} - s_{i,j,n} = C_j\delta_{i,n} - z_{j,n}$, and so $\mathbb{E}(C_j\hat{x}_{i,n} - s_{i,j,n}) = C_j\mathbb{E}(\delta_{i,n})$.  This implies $\alim_n \mathbb{E}(C_j\hat{x}_{i,n} - s_{i,j,n}) = 0$, which proves $\alim_n \mathbb{E}(\psi_n) = 0$ since the $e_i$ have zero mean.

Next observe the upper block triangle of (\ref{eqn:acttestjoint}) is correct by construction of $Q_{I,J,k'_{i,j}}$ and by definition of the $e_i$, and so we only have to prove that the lower-right block is correct.  In particular, note that $\mathbb{E}((C_j\delta_{i,n}-z_n)^{\vphantom{\textsf{T}}}(C_j\delta_{i,n}-z_n)^{\textsf{T}}) = \mathbb{E}((C_j\delta_{i,n})^{\vphantom{\textsf{T}}}(C_j\delta_{i,n})^{\textsf{T}}) + \Sigma_{Z,I}$ since $z_{i,n}$ is independent of $\delta_{i,n}$ by (\ref{eqn:dynamicschange}).  This implies that we have that $\alim_n \mathbb{E}((C_j\delta_{i,n}-z_n)^{\vphantom{\textsf{T}}}(C_j\delta_{i,n}-z_n)^{\textsf{T}})  = C_jD_iC_j^\textsf{T}+\Sigma_{Z,I}$.
\end{proof}

This result means that asymptotically the summation $S_{i,j} = \frac{1}{\ell}\sum_{n+1}^{n+\ell}\psi_{n,i,j}^{\vphantom{\textsf{T}}}\psi_{n,i,j}^\textsf{T}$ with $\ell \geq (m_i+q_i)$ has a Wishart distribution with $\ell$ degrees of freedom and a scale matrix that matches (\ref{eqn:acttestjoint}), and we use this to define a statistical test.  In particular, we check if the negative log-likelihood 
\begin{multline}
\label{eqn:nlltest}
\textstyle\mathcal{L} = \sum_{j=1}^\kappa (1-\ell+m_i+q_i)\cdot\log\det S_{n,i,j} + \\
\textstyle\sum_{j=1}^\kappa \trace\left(R_{I,J}^{-1}\cdot S_{n,i,j}\right)
\end{multline}
corresponding to this Wishart distribution and the summations of $S_{n,i,j}$ is large by conducting the hypothesis test
\begin{equation}
\begin{cases}
\text{reject, } &\text{if } \mathcal{L}(S_n) > \tau(\alpha)\\
\text{accept, } &\text{if } \mathcal{L}(S_n) \leq \tau(\alpha)
\end{cases}
\end{equation}
where $\tau(\alpha)$ is a threshold that controls the false error rate $\alpha$.  A rejection corresponds to the detection of an attack, while an acceptance corresponds to the lack of detection of an attack.  This notation emphasizes the fact that achieving a specified false error rate $\alpha$ (a false error in our context corresponds to detecting an attack when there is no attack occurring) requires changing the threshold $\tau(\alpha)$.

\section{Designing the State-Feedback}

\begin{algorithm}[t]
\caption{Compute State-Feedback $K$ for Proposition \ref{prop:nii}}
\label{alg:nii}
\begin{algorithmic}
\State choose any $v$ such that $Bv \in \cap_{i=1}^\kappa \range(B_i)$
\State compute $K'$ satisfying Heymann's lemma for $Bv$
\State compute $G$ such that $A+BK'+BvG$ is Schur stable
\State $K := K' + vG$
\end{algorithmic}
\end{algorithm}

\label{sec:sfbk}

We provide two approaches for designing a state-feedback controller that satisfies Condition \ref{cond:one}.  The first applies when $B$ is square (i.e., $\sum_{i=1}^k q_i = p$); though it generalizes to skinny $B$ (i.e., $\sum_{i=1}^k q_i < p$) in some cases, we do not prove this.  The first approach relies upon a multi-input generalization (which we construct and prove) of Heymann's lemma \cite{heymann1968comments,hautus1977simple}.  The second approach applies to $B$ of arbitrary size where the range spaces of $B_i$ have a nonempty intersection.

\subsection{Multiple Input Heymann's Lemma}

Heymann's lemma \cite{heymann1968comments,hautus1977simple} is used to prove arbitrary pole placement of controllable, multiple input LTI systems by allowing a reduction to the case of arbitrary pole placement of a controllable, single input LTI system.  Formally, it says

\begin{lemma}[Heymann's Lemma]
If $(A,B)$ is controllable, then for any $b = Bv \neq 0$ there exists $K$ (that depends on $b$) such that $(A+BK, b)$ is controllable.
\end{lemma}

We need a multiple input generalization of Heymann's Lemma.  Let $b_i$ denote the $i$-th column of the matrix $B$.  Then

\begin{proposition}
\label{prop:mihl}
If $B$ is full rank and square-shaped (i.e., $B\in\mathbb{R}^{p\times p}$); then there exists a single $K$ such that $A+BK$ is Schur stable and $(A+BK, b_i)$ is controllable for all $i$.
\end{proposition}

\begin{proof}
We prove this result stepwise.  Since $B$ is full rank and square-shaped, its columns are linearly independent.  Consider any $\lambda$ with $0 < |\lambda|< 1$, and define $x_1 = \frac{b_1}{\|b_1\|}$ and $x_{n+1} = Ax_n + Bu_n$.  Now suppose there exists $u_1,\ldots,u_{k-1}$ such that $x_i = \lambda^{i-1}\frac{b_i}{\|b_i\|}$ for $i = 1,\ldots,k$.  If $k < p$, then there exists a $u_k$ satisfying $\lambda^k\frac{b_{k+1}}{\|b_{k+1}\|} - \lambda^{k-1}A\frac{b_k}{\|b_k\|} = Bu_k$ since $B$ is full rank.  Hence by definition of the dynamics on $x$ there exists $u_k$ such that $x_{k+1} = \lambda^k\frac{b_{k+1}}{\|b_{k+1}\|}$.  If $k = p$, then there exists a $u_p$ satisfying $\lambda^p\frac{b_{1}}{\|b_1\|} - \lambda^{p-1}A\frac{b_p}{\|b_p\|} = Bu_p$ since $B$ is full rank.  So by definition of the dynamics on $x$ there exists $u_p$ such that $x_{p+1} = \lambda^p\frac{b_1}{\|b_1\|}$.  

Next define the matrices
\begin{align}
&U = \begin{bmatrix} u_1 & \cdots & u_p \end{bmatrix}\\
&R = \begin{bmatrix} \frac{b_1}{\|b_1\|} & \cdots & \frac{b_p}{\|b_p\|} \end{bmatrix}\\
&\Lambda = \diag\big(1, \lambda, \ldots,\lambda^{p-1}\big)
\end{align}
and $K = UR^{-1}\Lambda^{-1}$.  The matrices $\Lambda$ and $R$ are invertible by construction since $0 < |\lambda|< 1$ and $B$ is invertible.  Note that by definition $U = K\Lambda B$.  Finally, note for any $i$ we have
\begin{multline}
\label{eqn:25}
\|b_i\|\cdot\begin{bmatrix} \frac{b_i}{\|b_i\|} & \cdots & \frac{b_p}{\|b_p\|} & \frac{b_1}{\|b_1\|} & \cdots & \frac{b_{i-1}}{\|b_{i-1}\|}\end{bmatrix}\cdot\Lambda = \\
\begin{bmatrix} b_i & (A+BK)b_i & \cdots & (A+BK)^{p-1}b_i\end{bmatrix}.
\end{multline}
The left side has full rank by the assumptions on $B$, and the right side is the observability matrix for the $(A+BK, b_i)$.  This proves $(A+BK, b_i)$ is observable for all $i$ since we have shown that the observability matrix has full rank.

We conclude by proving that the above designed $K$ makes $A+BK$ Schur stable.  Consider any $x\in\mathbb{R}^p$, and observe that by the assumptions on $B$ there exists $z\in\mathbb{R}^p$ such that $x = Rz$.  But, as in (\ref{eqn:25}), by construction $(A+BK)^p\frac{b_i}{\|b_i\|} = \lambda^p\frac{b_i}{\|b_i\|}$ for all $i$.  Hence $(A+BK)^px = (A+BK)^p Rz = \lambda^p Rz = \lambda^p x$ for all $x$.  This means all the eigenvalues of $(A+BK)^p$ are $\lambda^p$, and using the spectral mapping theorem implies the eigenvalues of $A+BK$ are roots of $\lambda^p$.  Thus the magnitude of the eigenvalues of $A+BK$ are $|\lambda|$, which means that $A+BK$ is Schur stable since $0 < |\lambda| < 1$.
\end{proof}

Though the above is an existence result, a state-feedback matrix $K$ satisfying Proposition \ref{prop:mihl} can be computed using Algorithm \ref{alg:mihl}.  The correctness of this algorithm follows from the construction used in the proof of Proposition \ref{prop:mihl}.  Also, the next result proves that this $K$ satisfies Condition \ref{cond:one}.
%


\begin{corollary}
Suppose $C_j \neq 0$ for all $j$.  If $B$ is full rank and square-shaped (i.e., $B\in\mathbb{R}^{p\times p}$); then there exists a $K$ such that $A+BK$ is Schur stable and that Condition \ref{cond:one} holds.
\end{corollary}

\begin{proof}
Consider any $i \in \{1,\ldots,\kappa\}$, and choose $s$ to be any index such that the $s$-th column in $B$ belongs to $B_i$.  Proposition \ref{prop:mihl} says $(A+BK, b_s)$ is controllable.  This means the controllability matrix
\begin{equation}
\mathfrak{C}' = \begin{bmatrix} b_s & (A+BK)b_s & \ldots & (A+BK)^{p-1}b_s\end{bmatrix}
\end{equation}
has $\rank(\mathfrak{C}') = p$, and so the controllability matrix
\begin{equation}
\mathfrak{C} = \begin{bmatrix} B_i & (A+BK)B_i & \ldots & (A+BK)^{p-1}B_i\end{bmatrix}
\end{equation}
also has $\rank(\mathfrak{C}) = p$ since the columns of $\mathfrak{C}$ are a superset of the columns of $\mathfrak{C}'$.  Thus by Sylvester's rank inequality, we have $\rank(C_j\mathfrak{C}) \geq \rank(C_j) + \rank(\mathfrak{C}) - p = \rank(C)$. But $\rank(C_j) \geq 1$ since $C_j\neq 0$.  Combining this with the earlier inequality gives $\rank(C_j\mathfrak{C}) \geq 1$, and so $C_j\mathfrak{C} \neq 0$.  This means $k_{i,j}' \leq p-1$ exists since $C_j\mathfrak{C}$ is a block matrix consisting of the blocks $C_j(A+BK)^kB_i$.
\end{proof}

\subsection{Nonempty Intersection of Inputs}

We next consider $B$ with arbitrary shape, such that the range spaces of $B_i$ have a nonempty intersection.  Our Algorithm \ref{alg:nii} designs a $K$ for this case, and it uses Heymann's lemma \cite{heymann1968comments,hautus1977simple}.  The next result proves its correctness.

\begin{proposition}
\label{prop:nii}
Suppose $C_j \neq 0$ for all $j$ and that we have $\cap_{i=1}^\kappa \range(B_i) \neq \emptyset$.  If $(A,B)$ is controllable, then Algorithm \ref{alg:nii} computes a $K$ such that $A+BK$ is Schur stable and that Condition \ref{cond:one} is satisfied.
\end{proposition}

\begin{proof}
First note that $v$ exists by assumption, and so we can compute $K'$ by Heymann's lemma.  This means $(A+BK', Bv)$ is controllable, which implies that we can compute $G$ such that $A+BK'+BvG$ is Schur stable.  This proves that $A+BK$ is Schur stable since $K = K' + vG$.  Next consider any $i\in\{1,\ldots,\kappa\}$.  Since $(A+BK', Bv)$ is controllable, this means that $(A+BK'+ BvG, Bv)$ is controllable.  (This uses the well known fact that state-feedback does not affect controllability.)  As a result, we have
\begin{equation}
\mathfrak{C}' = \begin{bmatrix} Bv & (A+BK)Bv & \ldots & (A+BK)^{p-1}Bv\end{bmatrix}
\end{equation}
has $\rank(\mathfrak{C}') = p$.  But by assumption, there exists $v_i$ such that $B_iv_i = Bv$.  So if we define
\begin{equation}
\mathfrak{C} = \begin{bmatrix} B_i & (A+BK)B_i & \ldots & (A+BK)^{p-1}B_i\end{bmatrix},
\end{equation}
then we have that $\mathfrak{C}' = \mathfrak{C}\blkdiag(v_i, \ldots, v_i)$. Thus $p \geq \rank(\mathfrak{C}) \geq \rank(\mathfrak{C}') = p$. Thus by Sylvester's rank inequality, we have $\rank(C_j\mathfrak{C}) \geq \rank(C_j) + \rank(\mathfrak{C}) - p = \rank(C)$. But $\rank(C_j) \geq 1$ since $C_j\neq 0$.  Combining this with the earlier inequality gives $\rank(C_j\mathfrak{C}) \geq 1$, and so $C_j\mathfrak{C} \neq 0$.  This means $k_{i,j}' \leq p-1$ exists since $C_j\mathfrak{C}$ is a block matrix consisting of the blocks $C_j(A+BK)^kB_i$.
\end{proof}

\section{Simulation: Autonomous Vehicle Platooning}

\label{sec:sav}

\begin{figure}
\includegraphics[trim={0in 0.1in 0in 0.1in},clip,scale=0.9]{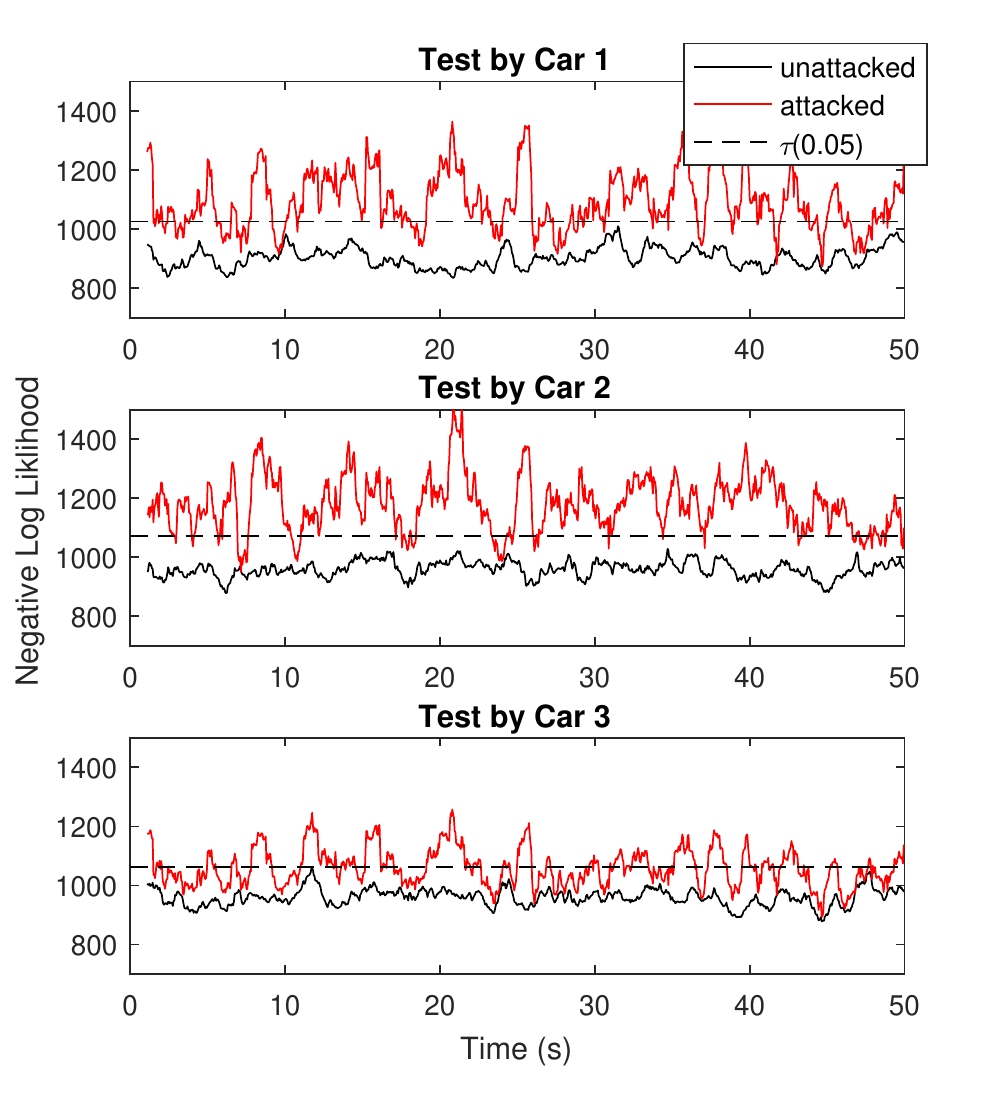}
\caption{Value of (\ref{eqn:nlltest}) for Simulation of Vehicle Platoon, with a Negative Log-Likelihood Threshold for $\alpha=0.05$ False Detection Error Rate\label{fig:tests}}
\end{figure}

We apply a standard model \cite{ulsoy2012automotive} for error kinematics of speed control of vehicles to generate a model for a three car platoon.  The state vector for the three car platoon is $x^\textsf{T} = \big[e_1\ d_1\ e_2\ d_2\ e_3\ d_3\big]$, where $e_i$ is deviation from the desired velocity for car $i$, and $d_i$ is deviation from the desired following distance between car $i$ and car $i+1$.  The discretized dynamics for a timestep of 0.05 seconds are then:
\begin{equation}
A = \begin{bmatrix} 1 & 0 & 0 & 0 & 0\\ -\frac{1}{20} & 1 & \frac{1}{20} & 0 & 0\\ 0 & 0 & 1 & 0 & 0\\ 0 & 0 & -\frac{1}{20} & 1 & \frac{1}{20} \\ 0 & 0 & 0 & 0 & 1\end{bmatrix}
\end{equation}
and
\begin{equation}
B_1 = \begin{bmatrix} \frac{1}{20} \\ -\frac{1}{800} \\ 0 \\ 0 \\ 0 \end{bmatrix} \quad B_2 = \begin{bmatrix} 0 \\ \frac{1}{800} \\ \frac{1}{20} \\ -\frac{1}{800} \\ 0\end{bmatrix} \quad B_3 = \begin{bmatrix} 0 \\ 0 \\ 0 \\ \frac{1}{800} \\ \frac{1}{20}\end{bmatrix}.
\end{equation}
Assuming each car measures its own velocity and the distance to the car in front of it, we have
\begin{equation}
\begin{aligned}
&C_1 = \begin{bmatrix} 1 & 0 & 0 & 0 & 0\end{bmatrix}\\
&C_2 = \begin{bmatrix} 0 & 1 & 0 & 0 & 0 \\ 0 & 0 & 1 & 0 & 0\end{bmatrix}\\
&C_2 = \begin{bmatrix} 0 & 0 & 0 & 1 & 0 \\ 0 & 0 & 0 & 0 & 1\end{bmatrix}
\end{aligned}
\end{equation}
We assume that the process and measurement noise had variance $\Sigma_W = 5\times 10^{-5}\cdot\mathbb{I}$ and $\Sigma_{Z,I} = 10^{-3}\cdot\mathbb{I}$, respectively.

We applied our statistical test (\ref{eqn:nlltest}) with each car using a watermarking signal with variance $\Sigma_{E,I} = 0.2$, where
\begin{equation}
\begin{aligned}
&K_1 = \begin{bmatrix} -1 & 0.1 & 0 & 0 & 0\end{bmatrix}\\
&K_2 = \begin{bmatrix} 1 & -1 & -2 & 0.1 & 0\end{bmatrix}\\
&K_3 = \begin{bmatrix} 0.5 & -0.5 & 0.5 & -1 & -2\end{bmatrix}
\end{aligned}
\end{equation}
and
\begin{equation}
L_1 = \begin{bmatrix} -0.5 \\ 0 \\ 0 \\ 0 \\ 0\end{bmatrix}\,  L_2 = \begin{bmatrix} 0.05 & 0 \\ -0.5 & 0 \\ 0 & -0.5\\ 0 & 0\\ 0& 0\end{bmatrix}\,  L_3 = \begin{bmatrix} 0 & 0 \\ 0 & 0 \\ 0.05 & 0 \\ -0.5 & 0 \\ 0 & -0.5\end{bmatrix}
\end{equation}
We conducted two simulations, where the platoon was un-attacked and attacked.  In the attack simulations, the attacker chose $v_{1,n}$ and $\nu_{2,3,n}$ to be zero mean i.i.d. jointly Gaussian random variables with variance $0.5$ and $0.2\cdot\mathbb{I}$, respectively. Fig. \ref{fig:tests} shows the results of applying our statistical test (\ref{eqn:nlltest}), and the plots show that our statistical watermarking test can detect the presence or absence of an attack.

\section{Conclusion}

This paper developed statistical watermarking to detect sensor and communication attacks on networked LTI systems. Unlike the non-networked case, watermarking in the networked case requires a $K$ so that $A+BK$ is controllable with respect to each $B_i$.  This ensures the private watermarking signal of each subcontroller is seen in the output of all subcontrollers.  We provided two algorithms to compute such a controller.  The efficacy of our watermarking approach was demonstrated by a simulation of a three car platoon.  One possible direction for future work is to explore the performance of our watermarking test under specific types of attacks, such as replay attacks or network disturbances.






\bibliographystyle{IEEEtran}
\bibliography{IEEEabrv,waten}

\end{document}